\documentclass[12pt]{article}
\usepackage{ejorsj-s2}
\usepackage{amsmath,amssymb,amsfonts,amsthm}
\usepackage{bm}
\usepackage[dvipdfmx]{graphicx}

\eqswitchon 
\newtheorem{theorem}{Theorem}[section]
\newtheorem{lemma}{Lemma}[section]

\newtheorem{example}{Example}[section]
\def\R{\mathbb{R}}

\def\eps{\varepsilon}

\title{AN APPLICATION OF BORSUK-ULAM'S THEOREM TO 
NONLINEAR PROGRAMMING}
\author{
\begin{tabular}[h]{c}  
Hidefumi Kawasaki  \\ 
\textit{Kyushu University}\\ 
\end{tabular}
}
\date{(August 26, 2023)}
\begin{document}
\maketitle
\begin{abstract}
Borsuk-Ulam's theorem is a useful tool of algebraic topology. 
It states that for any continuous mapping $f$ from the 
$n$-sphere to the $n$-dimensional Euclidean space, 
there exists a pair of antipodal points such that $f(x)=f(-x)$. 
As for its applications, ham-sandwich theorem, necklace theorem 
and coloring of Kneser graph by Lov\'{a}sz are well-known. 
This paper attempts to apply Borsuk-Ulam's theorem to nonlinear programming.
\end{abstract}
\section{Introduction}

Borsuk-Ulam's theorem \cite{Borsuk33} is an important theorem 
of algebraic topology. It states that for any continuous mapping $f$ 
from the $n$-sphere $S^n$ to the Euclidean space $\R^n$, 
there exists a point $x\in S^n$ such that $f(x)=f(-x)$. 
As for its applications, ham-sandwich theorem, necklace theorem 
and coloring of Kneser graph by Lov\'{a}sz \cite{Lovasz78} are well-known. 
It has several equivalent statements: 
Tucker's lemma is a combinatorial version, 
LSB theorem is a set-cover version, 
see e.g. Matou$\check{s}$ek~\cite{Mantousek2008}. 
This is reminiscent of Brouwer's fixed point theorem, 
which also has equivalent statements: 
Sperner's lemma is a combinatorial version, 
KKM lemma is a set-cover version. 
Borsuk-Ulam's theorem implies Brouwer's fixed point theorem. 
Howerver the converse is unknown. In this sense, 
Borsuk-Ulam's theorem seems stronger than Brouwer's fixed point theorem. 

Ham-sandwich theorem is one of the most famous applications 
of Borsuk-Ulam's theorem. 
Let $\mu$ be the Lebesgue measure of $\R^n$, 
and $A_i\subset\R^n\ (i=1,\dots,n)$ 
be compact sets with positive Lebesgue measure. 
Then ham-sandwich theorem states that there is a hyperplane $H$ which 
divides each $A_i$ in half w.r.t. the Lebesgue measure, that is,  
\[
	\mu(A_i\cap H_+)=\mu(A_i\cap H_-)\ (i=1,\dots,n),
\]
where $H_+$ and $H_-$ denote two closed half spaces determined by $H$. 
There are several versions of ham-sandwich theorem for finite points sets 
besides this statement.  

This paper aims to apply Borsuk-Ulams's theorem to nonlinear programming problems. 
In Section 2, we introduce a family of nonlinear programming problems 
with parameter $\bm u\in S^n$, and discuss the continuity 
of its optimal-value function w.r.t. $\bm u$. 
In Section 3, we apply Borsuk-Ulam's theorem to the optimal-value function. 
In Section 4, we sharpen our results. 

Throughout this paper, ${\rm int}\,A_i$ denotes the interior of $A_i$ and 
${\rm co}\,A_i$ denotes the convex hull of $A_i$. 

\section{Parametrized Nonlinear Programming Problem}  

In this section, we introduce a family of parametrized nonlinear programming problems, 
and discuss the continuity of optimal-value functions. 

For any point $\bm u=(u_1,\dots,u_n,u_{n+1})$ of $S^n$, 
we write $u=(u_1,\dots,u_n)\in\R^n$ and $\bm u=(u,u_{n+1})$. 
We assign to $\bm u\in S^n$ a hyperplane 
$H_{\bm u}=\{x\in\R^n\mid \langle u,x\rangle=u_{n+1}\}$ 
and two closed half-spaces: 
\[
	H_{\bm u}^+=\{x\in\R^n\mid \langle u,x\rangle\ge u_{n+1}\},\ 
	H_{\bm u}^-=\{x\in\R^n\mid \langle u,x\rangle\le u_{n+1}\},
\]
where $\langle u,x\rangle$ denotes the inner product $u_1x_1+\cdots +u_n x_n$.  
Then it is obvious that $H_{-\bm u}^+=H_{\bm u}^-$. 
In the case of $u\ne \bm 0$, both $H_{\bm u}^+$ and $H_{\bm u}^-$ are non-empty. 
In the case of $u=\bm 0$, one of $H_{\bm u}^+$ and $H_{\bm u}^-$ 
is $\R^n$, and the other is empty. 

Let $A_i$ be a non-empty compact convex subset of $\R^n$ 
and $f_i:\R^n\times S^n\to \R$ be a continuous function for any $i=1,\dots,n$. 
We consider $n$-tuple of optimal-value functions 
$\psi=(\psi_1,\dots,\psi_n):S^n\to \R^n$ defined by
\begin{equation}\label{opt-value}
	\psi_i(\bm u):=\left\{
	\begin{array}{cc}
	\underset{x\in A_i\cap H_{\bm u}^+}{\max}f_i(x,\bm u)
	-\underset{x\in A_i\cap H_{\bm u}^+}{\min}f_i(x,\bm u) & 
	(A_i\cap H_{\bm u}^+\ne\emptyset),\\
	0 & (A_i\cap H_{\bm u}^+=\emptyset).
	\end{array}\right.
\end{equation}
This $\psi$ and its variation $\varphi$ defined by (\ref{opt-value2}) later 
play the central role in this paper.  
We will apply Borsuk-Ulam's theorem to $\psi:S^n\to \R^n$. 
In order to show that $\psi$ is continuous on the whole set $S^n$, 
we define set-valued mappings $S_i:S^n\to A_i$ $(i=1,\dots,n)$ by
\begin{equation}\label{S(u)}
	S_i(\bm u)=A_i\cap H_{\bm u}^+\ (\bm u\in S^n).
\end{equation}
It is apparent that $S_i(\bm u)$ is empty if and only if 
$\max_{x\in A_i} \langle u,x\rangle < u_{n+1}$. 

Let $S$ be a set-valued mapping from a non-empty set $U\subset \R^m$ to a 
non-empty compact set $A\subset\R^n$. 
Then $S$ is said to be \textit{upper semi-continuous at $\bm u$} 
if $x^k\in S(\bm u^k)$ and 
$(x^k,\bm u^k)\to (x,\bm u)$ implies $x\in S(\bm u)$. 
$S$ is said to be \textit{lower semi-continuous at $\bm u$} if 
for any point $x\in S(\bm u)$ and sequence $\bm u^k\to \bm u$, 
there exists a sequence $x^k\in S(\bm u^k)$ converging to $x$. 
$S$ is said to be \textit{continuous at $\bm u$} if it is u.s.c and l.s.c. at $\bm u$. 
Lemma \ref{lem:opt-value1} below on the continuity of optimal-value functions is 
well-known in nonlinear programming, see e.g. Fiacco \cite{Fiacco1983}.

\begin{lemma}\label{lem:opt-value1} 
Let $f:\R^n\times U\to \R$ be a continuous function, and 
$A\subset\R^n$ be a non-empty compact set. 
Assume that a set-valued mapping $S:U\to A$ is continuous at 
$\bar{\bm u}\in U$. 
Then the optimal-value functions $\underset{x\in S(\bm u)}{\max} f(x,\bm u)$ 
and $\underset{x\in S(\bm u)}{\min} f(x,\bm u)$ are 
continuous at $\bar{\bm u}$. 
\end{lemma}

\begin{lemma}\label{lem:Hu+} 
The set-valued mapping $S_i$ defined by (\ref{S(u)}) satisfies the followings:
\begin{enumerate}
\item[(1)] $S_i$ is upper semi-continuous at any $\bm u\in S^n$. 
\item[(2)] $S_i$ is continuous at $\bm u$ if 
${\rm int}\,A_i\cap H_{\bm u}^+$ is non-empty. 
\item[(3)] If $S_i(\bm u)$ is a singleton and ${\rm int}\,A_i\cap H_{\bm u}^+$ is empty, 
then for any converging sequence $\bm u^k$ to $\bm u$, 
the diameter of $S_i(\bm u^k)$ converges to $0$. 
\end{enumerate}
\end{lemma}
\begin{proof} 
We write $\bm u=(u,u_{n+1})$ and $\bm u^k=(u^k,u^k_{n+1})$. \\
(1) Assume that $x^k\in S_i(\bm u^k)$ and $(x^k,\bm u^k)\to (x,\bm u)$. 
Then $\langle u^k,x^k\rangle\ge u_{n+1}^k$ and $x^k\in A_i$. 
Taking $k\to\infty$, we have $\langle u,x\rangle\ge u_{n+1}$ and $x\in A_i$, 
so that $x\in S_i(\bm u)$. \\
(2) Assume that $x\in S_i(\bm u)$ and $\bm u^k\to \bm u$. 
In the case of $\langle u,x\rangle>u_{n+1}$, taking $x^k:=x$, we have
\[
	|\langle u^k,x^k\rangle-u_{n+1}^k-\langle u,x\rangle+u_{n+1}|\le 
	\parallel u^k-u\parallel \parallel x\parallel +|u_{n+1}^k-u_{n+1}|\to 0.
\]
Hence $\langle u^k,x\rangle-u_{n+1}^k>0$ for all sufficiently large $k$. 
Therefore we have obtained a converging sequence $x^k\in S_i(\bm u^k)$ to $x$.

In the case of $\langle u,x\rangle=u_{n+1}$, take a point 
$a\in {\rm int}\,A_i\cap H_{\bm u}^+$, then $\langle u,a\rangle\ge u_{n+1}$. 
Since $a$ is an interior point of $A_i$, 
$b:=a+\delta u$ belongs to ${\rm int}\,A_i$ for a sufficiently small $\delta>0$, and 
\begin{equation}\label{ub>u0}
	\langle u,b\rangle> u_{n+1}.
\end{equation} 
Since $\bm u^k=(u_{n+1}^k,u^k)$ converges to $\bm u=(u,u_{n+1})$, it holds that 
$\langle u^k,b\rangle > u_{n+1}^k$ for all sufficiently large $k$. 
Hence, if $\langle u^k,x\rangle < u_{n+1}^k$, then 
the line segment joining $x$ and $b$, say $[x,b]$, and 
$H_{\bm u^k}=\{y\in\R^n\mid \langle u^k,y\rangle=u_{n+1}^k\}$ 
intersect at only one point, say $y^k$. 
Since $A_i$ is convex, $y^k$ belongs to $A_i\cap H_{\bm u^k}\subset S_i(\bm u^k)$. 
Taking 
\[
	x^k:=\begin{cases}
	x & (\langle u^k,x\rangle\ge u_{n+1}^k),\\
	y^k & (\langle u^k,x\rangle < u_{n+1}^k),
	\end{cases}
\]
we have $x^k\in S_i(\bm u^k)$. Further, $x^k$ converges to $x$. 
If not, we see from the compactness of $A_i$ that $y^k$ has a 
converging subsequence to some $y\ne x$. 
Taking $k\to \infty$ in $\langle u^k,y^k\rangle=u_{n+1}^k$, we have 
\begin{equation}\label{uy=u0}
	\langle u,y\rangle=u_{n+1}. 
\end{equation}
On the other hand, since sequence $y^k$ lies on the line segment $[x,b]$, 
so does $y$. 
It follows from (\ref{ub>u0}), (\ref{uy=u0}) and $\langle u,x\rangle=u_{n+1}$ that 
$y$ has to coincides with $x$. Thus we have obtained a converging sequence 
$x^k\in S_i(\bm u^k)$ to $x$. Therefore $S_i$ is l.s.c. at $\bm u$, so is 
continuous at $\bm u$.

(3) 
Deny the conclusion, then there exist a sequence $\bm u^k$ converging to 
$\bm u$ and $\delta>0$ such that ${\rm diam}\,S_i(\bm u^k)\ge\delta$. 
That is, there exist $y^k,\,z^k\in S(\bm u^k)$ such that 
\begin{equation}\label{y-z}
	\parallel y^k-z^k\parallel\ge\delta.
\end{equation} 
By compactness of $A_i$, we may assume that $y^k$ and $z^k$ converge to some 
$y,\,z\in A_i$, respectively. Taking $k\to \infty$ in (\ref{y-z}) and
\begin{equation}\label{ukyk}
	\langle u^k,y^k\rangle\ge u_{n+1}^k,\ \langle u^k,z^k\rangle\ge u_{n+1}^k,
\end{equation} 
we have $\parallel y-z\parallel\ge\delta$, $\langle u,y\rangle\ge u_{n+1}$, and 
$\langle u,z\rangle\ge u_{n+1}$. 
Therefore, $S_i(\bm u)$ includes distinct points $y$ and $z$, 
which contradicts the assumption of (3).
\end{proof}

\begin{theorem}\label{thm:continuous}
Let $A_i\subset \R^n$ be a compact convex set whose interior is 
non-empty and $f_i:\R^n\times S^n\to \R$ be continuous for any $i=1,\dots,n$.  
Assume that $A_i\cap H$ is a singleton for any boundary point $x$ of $A_i$ 
and for any supporting hyperplane $H$ of $A_i$ at $x$ for $i=1,\dots,n$. 
Then $\psi$ is continuous on the whole set $S^n$.
\end{theorem}
\begin{proof}
It follows from Lemma \ref{lem:opt-value1} and Lemma \ref{lem:Hu+} (2) that 
$\psi_i$ is continuous at $\bm u$ if ${\rm int}\,A_i\cap H_{\bm u}^+$ is non-empty. 
Otherwise, $S_i(\bm u)=A_i\cap H_{\bm u}^+$ is at most a singleton by 
the assumption of the theorem. Hence $\psi_i(\bm u)=0$ by definition. 
Now, let $\bm u^k\in S^n$ be a sequence converging to $\bm u$. 

(i) When $S_i(\bm u)$ is empty, $S_i(\bm u^k)$ is also empty 
for any sufficiently large $k$. 
Hence $\psi_i(\bm u^k)=0=\psi_i(\bm u)$, so that $\psi_i$ is continuous at $\bm u$. 

(ii) When $S_i(\bm u)$ is a singleton, it holds that 
\[
	\psi_i(\bm u)=\underset{x\in S_i(\bm u)}{\max} f_i(x,\bm u)
	-\underset{x\in S_i(\bm u)}{\min} f_i(x,\bm u)=0.
\]
If $\psi_i(\bm u^k)$ dose not converge to $\psi_i(\bm u)=0$, 
then we may assume that $\psi_i(\bm u^k)\ne 0$. Hence 
$S_i(\bm u^k)$ is non-empty and
\[
	\psi_i(\bm u^k)=\underset{x\in S_i(\bm u^k)}{\max} f_i(x,\bm u^k)
	-\underset{x\in S_i(\bm u^k)}{\min} f_i(x,\bm u^k).
\]
Further, it follows from Lemma \ref{lem:Hu+} (3) that the diameter of $S_i(\bm u^k)$ 
converges to $0$. Since $f_i$ is uniformly continuous on 
the compact set $A_i\times S^n$, it holds that for any $\varepsilon>0$, 
\[
	|f_i(y,\bm u^k)-f_i(z,\bm u^k)|<\eps\quad (y,\,z\in S_i(\bm u^k))
\]
for all sufficiently large $k$. Therefore 
\[
	0\le \underset{x\in S_i(\bm u^k)}{\max} f_i(x,\bm u^k)
	-\underset{x\in S_i(\bm u^k)}{\min} f_i(x,\bm u^k)<\eps,
\]
that is, $|\psi_i(\bm u^k)-\psi_i(\bm u)|=|\psi_i(\bm u^k)|<\eps$, 
which is a contradiction. 
\end{proof}

\section{Modification of the Optimal-Value Functions}  

In this section, we calculate the optimal-value function $\psi$ for a simple example, 
and show that $\psi$ has a trivial pair of antipodal points. 
We modify $\psi$ to avoid having trivial pair. 

\begin{example}\label{exam:BU1} 
$n=1$. 
Define $f:\R\times S^1\to \R$ by $f(x_1,\bm u)=u_1x_1$, and take $A=[-1,1]$. 
Representing $\bm u\in C^1$ in polar coordinates as $\bm u=(\cos\theta, \sin\theta)$ 
$(-\pi/4\le \theta\le 7\pi/4)$, we have 
\[
		S(\bm u) = \{x_1\in [-1,1]\mid x_1 \cos\theta\ge \sin\theta\}
	= \begin{cases}
		\emptyset & (\theta=\pi/2)\\
		[-1,1] & (\theta=3\pi/2)\\
		[-1,1]\cap [\tan\theta,\infty) & (\cos\theta>0)\\
		[-1,1]\cap (-\infty,\tan\theta] & (\cos\theta<0).
	\end{cases}
\]
Hence $S(\bm u)=\emptyset\ \Leftrightarrow\ \tan\theta\notin [-1,1]\ 
\Leftrightarrow\ \pi/4<\theta<3\pi/4$. Otherwise, 
\[
		S(\bm u) = \begin{cases}
		[\tan\theta,1] & (-\pi/4\le\theta\le\pi/4)\\
		[-1,\tan\theta] & (3\pi/4\le\theta\le 5\pi/4)\\
		[-1,1] & (5\pi/4\le\theta\le 7\pi/4).
	\end{cases}
\]
By elementary calculation, we have
\[
	\underset{x\in S(\bm u)}{\max}f(x,\bm u)-\underset{x\in S(\bm u)}{\min}f(x,\bm u)=
	\left\{\begin{array}{cl}
	\cos(\theta)-\sin(\theta) & (-\pi/4\le\theta\le\pi/4)\\
	\sin(\theta)-\cos(\theta) & (3\pi/4\le\theta\le5\pi/4)\\
	2|\cos(\theta)| & (5\pi/4,7\pi/4).
	\end{array}\right.
\]
Therefore the optimal-value function (\ref{opt-value}) turns into 
\[
	\psi(\bm u)=\left\{\begin{array}{cl}
	|\cos(\theta)-\sin(\theta)| & \mbox{on\ }[-\pi/4,\pi/4]\cup [3\pi/4,5\pi/4]\\
	0 & \mbox{on\ }[\pi/4,3\pi/4]\\
	2|\cos(\theta)| & \mbox{on\ }[5\pi/4,7\pi/4].
	\end{array}\right.
\]
\begin{figure}[htb]
	\centering\includegraphics[width=50mm]{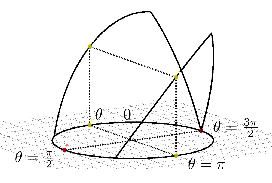}
	\caption{The graph of $\psi$ on $S^1$.}
	\label{Fig:bu1}
\end{figure}
\end{example}

In this example, 
there are two pairs of antipodal points where $\psi$ has the same value. 
One is $\psi(\bm u)=\psi(-\bm u)=1$ for $\bm u=(1,0)$ corresponding to 
$\theta=0,\,\pi$.  
The other is $\psi(\bm u^*)=\psi(-\bm u^*)=0$ for $\bm u^*=(0,1)$ 
corresponding to $\theta=\pi/2,\,3\pi/2$. We examine them. 
\begin{itemize}
\item 
For $\bm u=(1,0)$, since $S(\bm u)=A\cap H_{\bm u}^+=[0,1]$ and 
$S(-\bm u)=A\cap H_{-\bm u}^+=[-1,0]$, 
the equation $\psi(1,0)=\psi(-1,0)=1$ implies 
\begin{eqnarray*}
	\underset{x_1\in A\cap H_{\bm u}^+}{\max} u_1x_1
	-\underset{x_1\in A\cap H_{\bm u}^+}{\min} u_1x_1
	&=& \underset{x_1\in A\cap H_{-\bm u}^+}{\max}(-u_1x_1)
	-\underset{x_1\in A\cap H_{-\bm u}^+}{\min}(-u_1x_1)\\
	&=& \underset{x_1\in A\cap H_{\bm u}^-}{\max} u_1x_1
	-\underset{x_1\in A\cap H_{\bm u}^-}{\min} u_1x_1\\
	&=& 1.
\end{eqnarray*}

\item Since $S(\bm u^*)$ is empty for $\bm u^*=(0,1)$, 
$\psi(0,1)=0$ was artificially defined to ensure the continuity of $\psi$. 
Further, since $v_1=0$ and $A\cap H_{-\bm u^*}^+=A$, $\psi(0,-1)=0$ implies 
a trivial equation
\[
	\underset{x_1\in A\cap H_{-\bm u^*}^+}{\max}(-v_1x_1)
	-\underset{x_1\in A\cap H_{-\bm u^*}^+}{\min}(-v_1x_1)=
	\underset{x_1\in A}{\max}\, 0	-\underset{x_1\in A}{\min}\, 0=0.
\]
In this sense, $(\bm u^*,\,-\bm u^*)$ can be regarded as a trivial pair.
\end{itemize}

Now we modify $\psi$ to avoid having a trivial pair. 
Define $\varphi=(\varphi_1,\dots,\varphi_n):S^n\to \R^n$ by
\begin{equation}\label{opt-value1}
	\varphi_i(\bm u):=\left\{
	\begin{array}{cc}
	\underset{x\in A_i\cap H_{\bm u}^+}{\max}f_i(x,\bm u)
	-\underset{x\in A_i\cap H_{\bm u}^+}{\min}f_i(x,\bm u) & (A_i\cap H_{\bm u}^+\ne\emptyset),\\
	(1-\eps-u_{n+1})/\eps & (A_i\cap H_{\bm u}^+=\emptyset,\ u_{n+1}>1-\eps),\\
	0 & (A_i\cap H_{\bm u}^+=\emptyset,\ u_{n+1}\le 1-\eps),
	\end{array}\right.
\end{equation}
where $\eps>0$ is assumed to be sufficiently small. 
For the sake of convenience we define
\begin{equation}\label{eps-N}
	N_\eps:=\{\bm u\in S^n\mid u_{n+1}>1-\eps\},
\end{equation}
and call it \textit{$\eps$-neighborhood of the north pole $\bm u^*$}. 

\begin{lemma}\label{lem:varphi} 
Let $\eps>0$ be sufficiently small. Then under the assumptions of 
Theorem \ref{thm:continuous}, 
it holds for each $i$ that
\begin{enumerate}
\item[(1)] $\varphi_i$ takes the value $-1$ at the north pole 
$\bm u^*:=(0,\dots, 0,1)\in S^n$.  
\item[(2)] $A_i\cap H_{\bm u}^+$ is empty for any $\bm u\in N_\eps$. 
\item[(3)] $\varphi_i$ is $0$ on the boundary $\partial N_\eps$. 
\item[(4)] $\varphi_i$ is negative on $N_\eps$. 
\item[(5)] $\varphi_i$ is continuous on $S^n$. 
\end{enumerate}
\end{lemma}
\begin{proof} 
Since $A_i\cap H_{\bm u^*}^+=\{x\in A_i\mid \langle 0,x\rangle\ge 1\}=\emptyset$ 
for the north pole $\bm u^*$, and since $u^*_{n+1}=1$, 
(1) is apparent from the second case of (\ref{opt-value1}). 

(2) For any $\bm u=(u,u_{n+1})\in N_\eps$, it holds that
\[
	\langle u,x\rangle\le \parallel u\parallel \parallel x\parallel\le\sqrt{2\eps-\eps^2}M_i
	<1-\eps<u_{n+1},
\]
where $M_i>0$ is an upper bound of $\{\parallel x\parallel\ \mid x\in A_i\}$. 
Therefore $A_i\cap H_{\bm u}^+$ is empty. 

(3) Assume that $\bm u$ satisfies $u_{n+1}=1-\eps$. Then as well as (2), 
$A_i\cap H_{\bm u}^+$ is empty. 
So $\varphi_i(\bm u)=0$ follows from the third case of (\ref{opt-value1}). 

(4) Since $\varphi_i(u,u_{n+1})$ is linear w.r.t. $u_{n+1}\in [1-\eps,1]$, 
(4) follows from (1) and (3).  Since $\varphi_i=\psi_i$ outside of $N_\eps$, 
we get (5). 
\end{proof}

\begin{lemma}\label{lem:varphi2}
Assume that $f_i(x,\bm u)$ is antipodal w.r.t. $\bm u$, that is,  
$f_i(x,-\bm u)=-f_i(x,\bm u)$ for any $(x,\bm u)\in\R^n\times S^n$. Then 
\begin{equation}\label{opt-value2}
	\varphi_i(-\bm u)=\left\{
	\begin{array}{cc}
	\underset{x\in A_i\cap H_{\bm u}^-}{\max}f_i(x,\bm u)
	-\underset{x\in A_i\cap H_{\bm u}^-}{\min}f_i(x,\bm u) & (A_i\cap H_{\bm u}^-\ne\emptyset),\\
	(1-\eps +u_{n+1})/\eps & (A_i\cap H_{\bm u}^-=\emptyset,\ -u_{n+1}>1-\eps),\\
	0 & (A_i\cap H_{\bm u}^-=\emptyset,\ -u_{n+1}\le 1-\eps).
	\end{array}\right.
\end{equation}
\end{lemma}
\begin{proof} 
When $A_i\cap H_{-\bm u}^+$ is non-empty, since $H_{-\bm u}^+=H_{\bm u}^-$, we have
\begin{eqnarray*}
	\varphi_i(-\bm u) &=& \underset{x\in A_i\cap H_{-\bm u}^+}{\max}f_i(x,-\bm u)
	-\underset{x\in A_i\cap H_{-\bm u}^+}{\min}f_i(x,-\bm u)\\
	&=& \underset{x\in A_i\cap H_{\bm u}^-}{\max}-f_i(x,\bm u)
	-\underset{x\in A_i\cap H_{\bm u}^-}{\min}-f_i(x,\bm u)\\
	&=& \underset{x\in A_i\cap H_{\bm u}^-}{\max}f_i(x,\bm u)
	\underset{x\in A_i\cap H_{\bm u}^-}{-\min}f_i(x,\bm u).
\end{eqnarray*}
Otherwise, (\ref{opt-value2}) directly follows from the definition of $\varphi_i$. 
\end{proof}

By applying Borsuk-Ulam's theorem to $\varphi:S^n\to\R^n$, 
we obtain the following.

\begin{theorem}\label{thm:BU-opt1}
If the assumptions of Theorem \ref{thm:continuous} and Lemma \ref{lem:varphi2} 
are satisfied, then there exists some $\bm u\in S^n$ such that 
either of (i) or (ii) below holds for each $i=1,\dots,n$.  
\begin{enumerate}
\item[(i)] Both $A_i\cap H_{\bm u}^+$ and $A_i\cap H_{\bm u}^-$ are non-empty, 
and
\begin{equation}\label{max-min1}
	\underset{x\in A_i\cap H_{\bm u}^+}{\max}f_i(x,\bm u)
	-\underset{x\in A_i\cap H_{\bm u}^+}{\min}f_i(x,\bm u)
	=\underset{x\in A_i\cap H_{\bm u}^-}{\max}f_i(x,\bm u)
	-\underset{x\in A_i\cap H_{\bm u}^-}{\min}f_i(x,\bm u).
\end{equation}
\item[(ii)] One of $A_i\cap H_{\bm u}^+$ and $A_i\cap H_{\bm u}^-$ is empty, 
and the other is $A_i$. Further $f_i(\cdot,\bm u)$ is constant on $A_i$. 
\end{enumerate}
\end{theorem}
\begin{proof}
By Borsuk-Ulam's theorem,  there exists some $\bm u\in S^n$ such that 
$\varphi_i(\bm u)=\varphi_i(-\bm u)$ for all $i=1,\dots,n$.  

(i) When both $A_i\cap H_{\bm u}^+$ and $A_i\cap H_{\bm u}^-$ are non-empty, 
(\ref{max-min1}) is a direct consequence of (\ref{opt-value1}) and (\ref{opt-value2}). 

(ii) When $A_i\cap H_{\bm u}^+=\emptyset$, it holds that $A_i\cap H_{\bm u}^-=A_i$. 
Suppose that $u_{n+1}>1-\eps$, then we get from 
(\ref{opt-value1}) and (\ref{opt-value2}) that 
\[
	0>\frac{1-\eps -u_{n+1}}{\eps}=\varphi_i(\bm u)=\varphi_i(-\bm u)=
	\underset{x\in A_i}{\max}f_i(x,\bm u)-\underset{x\in A_i}{\min}f_i(x,\bm u)\ge 0,
\]
which is a contradiction. Hence $u_{n+1}\le 1-\eps$. So by definition of $\varphi_i$, 
\[
	0=\varphi_i(\bm u)=\varphi_i(-\bm u)=
	\underset{x\in A_i}{\max}f_i(x,\bm u)-\underset{x\in A_i}{\min}f_i(x,\bm u),
\]
which implies that $f_i(\cdot,\bm u)$ is constant on $A_i$.

Similarly $f_i(\cdot,\bm u)$ is constant on $A_i$ when $A_i\cap H_{\bm u}^-$ 
is empty.    
\end{proof}

When we take a bilinear form as $f_i(x,\bm u)$, 
case (ii) of Theorem \ref{thm:BU-opt1} is removed. 

\begin{theorem}\label{thm:BU-opt0}
Assume that the assumptions of Theorem \ref{thm:continuous} are satisfied. 
Let $Q_i$ be a non-singular matrix of order $n$, and take 
$f_i(x,\bm u)=\langle u,Q_ix\rangle$ for any $i=1,\dots,n$. 
Then there exists some $\bm u\in S^n$ such that 
both $A_i\cap H_{\bm u}^+$ and $A_i\cap H_{\bm u}^-$ are non-empty, 
and
\begin{equation}\label{max-min0}
	\underset{x\in A_i\cap H_{\bm u}^+}{\max}f_i(x,\bm u)
	-\underset{x\in A_i\cap H_{\bm u}^+}{\min}f_i(x,\bm u)
	=\underset{x\in A_i\cap H_{\bm u}^-}{\max}f_i(x,\bm u)
	-\underset{x\in A_i\cap H_{\bm u}^-}{\min}f_i(x,\bm u).
\end{equation}
\end{theorem}
\begin{proof}
Since $\langle u,Q_ix\rangle$ is antipodal w.r.t. $\bm u$ for any $x$, 
Theorem \ref{thm:BU-opt1} is applicable, 
and (\ref{max-min0}) is nothing but case (i) of Theorem \ref{thm:BU-opt1}.  

Next we show that case (ii) of Theorem \ref{thm:BU-opt1} is impossible. 
Suppose that $\bm u$ satisfies (ii). 
Then $f_i(x,\bm u)=\langle u,Q_ix\rangle$ is constant on $A_i$. 
Since the interior of $A_i$ is non-empty, it implies that 
$u=\bm 0\in\R^n$, so that $\bm u=(\bm 0,\pm 1)$. 
In the case of $\bm u$ is the north pole $(\bm 0,1)$, 
we get $\varphi_i(\bm u)=-1$ from definition of $\varphi_i$. 
Further since $A_i\cap H_{-\bm u}^+=A_i$, we have 
\[
	\varphi_i(\bm u)=\varphi_i(-\bm u)=
	\underset{x\in A_i}{\max}\langle \bm 0,Q_ix\rangle-
	\underset{x\in A_i}{\min}\langle \bm 0,Q_i x\rangle=0,
\]
which is a contradiction. 
In the case of $\bm u=(\bm 0,-1)$, since $A_i\cap H_{\bm u}^+=A_i$, we have
\[
	\varphi_i(\bm u)=\underset{x\in A_i}{\max}\langle \bm 0,Q_i x\rangle-
	\underset{x\in A_i}{\min}\langle \bm 0,Q_i x\rangle=0.
\]
On the other hand, $\varphi_i(\bm u)=\varphi_i(-\bm u)=\varphi_i(\bm 0,1)=-1$, 
which is a contradiction. Therefore (ii) is impossible. 
\end{proof}

\section{Special case of $f_i(x,\bm u)$}

In Sections 2 and 3, we required the strict convexity of $A_i$ 
to guarantee the continuity of $\psi_i$ and $\varphi_i$. 
If we take $f_i(x,\bm u)=\langle u,x\rangle$, we do not need this assumption. 
This section focuses on $f_i(x,\bm u)=\langle u,x\rangle$. 

\begin{lemma}\label{lem:<ux>} 
When we take $f_i(x,\bm u)=\langle u,x\rangle$ for any $i=1,\dots,n$, 
$\varphi_i$ is continuous on the whole $S^n$ 
without assuming the strict convexity of $A_i$.  
\end{lemma}
\begin{proof}
(i) When ${\rm int}\,A_i\cap H_{\bm u}^+$ is non-empty, 
$A_i\cap H_{\bm u'}^+$ is also non-empty 
for any $\bm u'$ sufficiently close to $\bm u$. 
Hence $\varphi_i(\bm u')=\psi_i(\bm u')$ by definition, which implies that 
$\varphi_i$ is continuous at $\bm u$. 

(ii) When ${\rm int}\,A_i\cap H_{\bm u}^+$ is empty 
and $A_i\cap H_{\bm u}^+$ is non-empty, it holds that
\[
	A_i\cap H_{\bm u}^+=\{x\in A_i\mid \langle u,x\rangle\ge u_{n+1}\}
	=\{x\in A_i\mid \langle u,x\rangle=u_{n+1}\}=A_i\cap H_{\bm u}.
\]
Therefore
\[
	\varphi_i(\bm u) = 
	\underset{x\in A_i\cap H_{\bm u}^+}{\max}\langle u,x\rangle-
	\underset{x\in A_i\cap H_{\bm u}^+}{\min}\langle u,x\rangle=0.
\]
Since $\bm u$ is not the north pole, it is outside of the $\eps$-neighborhood $N_\eps$ 
of the north pole for sufficiently small $\eps>0$. 
Now assume that $\bm u^k$ converges to $\bm u$. 
Then we may assume that $\bm u^k$ is also outside of $N_\eps$. 
So $\varphi_i(\bm u^k)\ge 0$ by definition. 
Since we want to prove $\varphi_i(\bm u^k)$ converges to $\varphi_i(\bm u)=0$, 
we may ignore $\bm u^k$ such that $\varphi_i(\bm u^k)=0$. 
Otherwise, 
\begin{eqnarray}
	\varphi_i(\bm u^k) &=& 
	\underset{x\in A_i\cap H_{\bm u^k}^+}{\max}\langle u^k,x\rangle-
	\underset{x\in A_i\cap H_{\bm u^k}^+}{\min}\langle u^k,x\rangle\nonumber\\
	&=& 
	\underset{x\in A_i,\ \langle u^k,x\rangle\ge u_{n+1}^k}{\max}\langle u^k,x\rangle-
	\underset{x\in A_i,\ \langle u^k,x\rangle\ge u_{n+1}^k}{\min}\langle u^k,x\rangle
	\nonumber\\
	& \le &
	\underset{x\in A_i,\ \langle u^k,x\rangle\ge u_{n+1}^k}{\max}\langle u^k,x\rangle-
	u_{n+1}^k.\label{varphi-max}
\end{eqnarray}
Let $x=y^k\in A_i$ maximize RHS of (\ref{varphi-max}). Then
$0 < \varphi_i(\bm u^k)\le \langle u^k,y^k\rangle-u_{n+1}^k$. Hence
\[
	0 \le \limsup_{k\to\infty}\varphi_i(\bm u^k)\le 
	\limsup_{k\to\infty}(\langle u^k,y^k\rangle-u_{n+1}^k).
\]
By taking a subsequence of $y^k$, we may assume that 
\[
	\limsup_{k\to\infty}(\langle u^k,y^k\rangle-u_{n+1}^k)=
	\lim_{k\to\infty}(\langle u^k,y^k\rangle-u_{n+1}^k).
\]
Further, by taking a converging subsequence $y^k\to y$, we have
\[
	0 \le \lim_{k\to\infty}(\langle u^k,y^k\rangle-u_{n+1}^k)=\langle u,y\rangle-u_{n+1}.
\]
Hence $y\in A_i\cap H_{\bm u}^+=A_i\cap H_{\bm u}$, that is, 
$\langle u,y\rangle-u_{n+1}=0$. Therefore
\[
	0\le \liminf_{k\to\infty}\varphi_i(\bm u^k)
	\le \limsup_{k\to\infty}\varphi_i(\bm u^k)=0.
\]

(iii) When $A_i\cap H_{\bm u}^+$ is empty, 
$A_i\cap H_{\bm u'}^+$ is also empty for any $\bm u'$ sufficiently close to $\bm u$. 
Hence $\varphi_i(\bm u)$ and $\varphi_i(\bm u')$ are defined by 
the second or the third case of (\ref{opt-value1}). 
So $\varphi_i$ is continuous at $\bm u$. 
\end{proof}

\begin{theorem}\label{thm:BU-opt2}
Let $A_i\subset \R^n\ (i=1,\dots,n)$ be compact convex sets whose interiors are 
non-empty. Then there exists some $\bm u\in S^n$ such that 
both $A_i\cap H_{\bm u}^+$ and $A_i\cap H_{\bm u}^-$ are non-emopty and 
\begin{equation}\label{max-min2}
	\delta^*(u\mid A_i\cap H_{\bm u}^+)-\delta_*(u\mid A_i\cap H_{\bm u}^+)
	=\delta^*(u\mid A_i\cap H_{\bm u}^-)-\delta_*(u\mid A_i\cap H_{\bm u}^-) 
\end{equation}
for all $i=1,\dots,n$, 
where $\delta^*$ and $\delta_*$ denote the support function and the infimum 
support function:
\[
	\delta^*(u\mid X):=\underset{x\in X}{\max}\langle u,x\rangle,\ 
	\delta_*(u\mid X):=\underset{x\in X}{\min}\langle u,x\rangle,
\]
respectively. 
\end{theorem}
\begin{proof}
Take $f_i(x,\bm u)=\langle u,x\rangle$ for any $i=1,\dots,n$. 
Since $\varphi_i$ is continuous by Lemma \ref{lem:<ux>}, 
we obtain (\ref{max-min0}) as well as Theorem \ref{thm:BU-opt0}. 
Then LHS of (\ref{max-min0}) turns into 
\[
	\varphi_i(\bm u) = \underset{x\in A_i\cap H_{\bm u}^+}{\max}\langle u,x\rangle
	-\underset{x\in A_i\cap H_{\bm u}^+}{\min} \langle u,x\rangle
	=\delta^*(u\mid A_i\cap H_{\bm u}^+)-\delta_*(u\mid A_i\cap H_{\bm u}^+),
\]
and RHS of (\ref{max-min0}) turns into 
\[
	\varphi_i(-\bm u) = \underset{x\in A_i\cap H_{\bm u}^-}{\max}\langle u,x\rangle
	-\underset{x\in A_i\cap H_{\bm u}^-}{\min} \langle u,x\rangle
	=\delta^*(u\mid A_i\cap H_{\bm u}^-)-\delta_*(u\mid A_i\cap H_{\bm u}^-).
\]
Therefore we obtain (\ref{max-min2}). 
\end{proof}

We have just removed the assumption of the strict convexity of $A_i$. 
By taking its convex hull,  we do not need to require convexity either. 
So $A_i$ can be finite.  

\begin{theorem}\label{thm:BU-opt3}
Let $A_i\subset \R^n$ be a compact set 
whose convex hull has a non-empty interior for any $i=1,\dots,n$. 
Then there exists some $\bm u\in S^n$ such that 
both $A_i\cap H_{\bm u}^+$ and $A_i\cap H_{\bm u}^-$ are non-emopty and 
\begin{equation}\label{max-min3}
	\delta^*(u\mid {\rm co}\,A_i\cap H_{\bm u}^+)
	-\delta_*(u\mid {\rm co}\,A_i\cap H_{\bm u}^+)
	=\delta^*(u\mid {\rm co}\,A_i\cap H_{\bm u}^-)
	-\delta_*(u\mid {\rm co}\,A_i\cap H_{\bm u}^-)
\end{equation}
for all $i=1,\dots,n$.
\end{theorem}
\begin{proof}
This is a direct consequence of Theorem \ref{thm:BU-opt2}. 
\end{proof}

\section{Acknowledgements}

This research is supported by JSPS KAKENHI Grant Number 16K05278.

\newpage

\vspace{15pt}
{\leftskip 8.5cm \parindent 0mm
Hidefumi Kawasaki\\
Faculty of Mathematics\\
Kyushu University\\
Motooka 744, Nishi-ku\\
Fukuoka 819-0395, Japan\\ 
E-mail: \texttt{kawasaki.hidefumi.245@m.kyushu-u.ac.jp}
\par}

\begin{thebibliography}{99}

\bibitem{Borsuk33} K. Borsuk: 
Drei S\"atze \"uver die $n$-dimensionale euklidische Sph\"are, 
\textit{Fundamenta Mathematicae}, \textbf{20} (1933), 177--190.

\bibitem{Fiacco1983} A. V. Fiacco: 
\textit{Introduction to Sensitivity and Stability Analysis in Nonlinear Programming} 
(Academic Press, New York, 1983).

\bibitem{Lovasz78} L. Lov\'{a}sz, 
Kneser's conjecture, chromatic number and homotopy, 
\textit{Journal of Combinatorial Theorey, Ser. A}, \textbf{25} (1978), 319-324.

\bibitem{Mantousek2008} J. Matou$\check{s}$ek, 
\textit{Using the Borsuk-Ulam Theorem}, (Springer, Berlin Heidelberg, 2008).


\end{thebibliography}
\end{document}